\documentclass[12pt,a4paper,reqno]{amsart}
\usepackage{amsmath}
\usepackage{amsfonts}
\usepackage{amssymb}
\usepackage{bm} 

\newcommand\R{{\mathbf{R}}}
\newcommand\C{{\mathbf{C}}}

\newcommand\E{{\mathrm{E}}}
\newcommand\M{{\operatorname{M}}}

\newcommand\margin{{\operatorname{margin}}}
\newcommand\supp{{\operatorname{supp}}}

\newcommand\eps{{\varepsilon}}

\newcommand\dist{{\operatorname{dist}}}

\newcommand\red{{\operatorname{red}}}
\newcommand\blue{{\operatorname{blue}}}

\theoremstyle{plain}
  \newtheorem{theorem}[subsection]{Theorem}

  \newtheorem{proposition}[subsection]{Proposition}
  \newtheorem{lemma}[subsection]{Lemma}
  \newtheorem{corollary}[subsection]{Corollary}

\theoremstyle{remark}
  \newtheorem{remark}[subsection]{Remark}

\theoremstyle{definition}
  \newtheorem{definition}[subsection]{Definition}

\include{psfig}

\begin{document}

\title[Inverse theorem for bilinear wave]{An inverse theorem for the bilinear $L^2$ Strichartz estimate for the wave equation}
\author{Terence Tao}
\address{Department of Mathematics, UCLA, Los Angeles CA 90095-1555}
\email{tao@math.ucla.edu}

\vspace{-0.3in}
\begin{abstract} A standard bilinear $L^2$ Strichartz estimate for the wave equation, which underlies the theory of $X^{s,b}$ spaces of Bourgain and Klainerman-Machedon, asserts (roughly speaking) that if two finite-energy solutions to the wave equation are supported in transverse regions of the light cone in frequency space, then their product lies in spacetime $L^2$ with a quantitative bound.  In this paper we consider the \emph{inverse problem} for this estimate: if the product of two waves has large $L^2$ norm, what does this tell us about the waves themselves?  The main result, roughly speaking, is that the lower-frequency wave is dispersed away from a bounded number of light rays.  This result will be used in a forthcoming paper \cite{tao:heatwave4} of the author on the global regularity problem for wave maps.
\end{abstract}

\maketitle

\section{Introduction}

\subsection{Motivation}

Fix a dimension $n \geq 2$.  A \emph{wave} is defined to be a (tempered distributional) solution $\phi: \R \times \R^n \to \C$ to the wave equation
$$ - \phi_{tt} + \Delta \phi = 0,$$
thus the spacetime Fourier transform
$$ \tilde \phi(\tau,\xi) := \int_\R \int_{\R^d} \phi(t,x) e^{-2\pi i (t \tau + x \cdot \xi)}\ dx dt$$
of a wave (where the integrals should be interpreted in a distributional or limiting sense) is a measure 
$$ \tilde \phi(\tau,\xi) = f_+(\xi) \delta(\tau - |\xi|) + f_-(\xi) \delta(\tau + |\xi|)$$
supported on the light cone
$$ \Sigma := \{ (\tau,\xi) \in \R \times \R^d: |\tau| = |\xi| \},$$
where $\delta$ is the Dirac distribution and $f_+, f_-$ are measurable functions.  We will restrict attention to waves whose \emph{mass}
$$ \M( \phi ) := \int_{\R^d} |f_+(\xi)|^2 + |f_-(\xi)|^2\ d\xi$$ 
is finite.  From Plancherel's theorem we observe that
\begin{equation}\label{massif}
 \| \phi \|_{L^\infty_t L^2_x(\R \times \R^n)} \lesssim \M(\phi)^{1/2}
 \end{equation}
for all waves; here and in the sequel, we use $X \lesssim Y$ to denote the estimate $X \leq C_n Y$ for some constant $C_n$ depending on the dimension $n$ that varies from line to line, and $L^p_t L^q_x$ denotes the usual spacetime mixed Lebesgue norms.

Let $e_1,\ldots,e_n$ be the standard basis of $\R^d$.  We identify some special classes of waves:
\begin{itemize}
\item A wave $\phi$ has \emph{frequency $2^k$} if $\tilde \phi$ is supported on the conic annulus $\Sigma_k := \{ (\tau,\xi) \in \Sigma: 2^k \leq |\xi| \leq 2^{k+1} \}$.
\item A wave $\phi$ is \emph{red} if it is supported in the set $\Sigma_\red := \{ ( |\xi|, \xi ): \xi \in \R^d, \angle(\xi,e_1) \leq \pi/8 \}$.
\item A wave $\phi$ is \emph{blue} if it is supported in the set $\Sigma_\blue := \{ ( -|\xi|, \xi ): \xi \in \R^d, \angle(\xi,e_1) \leq \pi/8 \}$.
\end{itemize}

\begin{remark} If a wave $\phi$ has frequency $2^k$, then the conserved energy
$$ \E(\phi) = \frac{1}{2} \int_{\R^n} |\phi_t(t,x)|^2 + |\nabla \phi(t,x)|^2\ dx$$
is comparable to $2^k \M(\phi)$.  However, we will not use the energy in this paper.
\end{remark}

In the theory of nonlinear wave equations, it is of interest to estimate the product of two waves, taking advantage of any transversality of the frequency supports of such waves.  A model problem is that of estimating products of red waves $\phi$ and blue waves $\psi$ (the general case can be obtained from this special case by dyadic decomposition in frequency and angle, Lorentz transforms and conjugation, see e.g. \cite{tv:cone2}, \cite{tao:cone}, \cite{lv}, \cite{lrv}).  We may normalise $\phi$ to have frequency $1$ and $\psi$ to have frequency $2^k$ for some $k \geq 0$, thus the blue wave will have the higher frequency.

We have the following fundamental bilinear Strichartz estimate (see e.g. \cite{borg:cone}, \cite{kl-mac:null}, \cite{tv:cone1}, \cite{mock:cone}, \cite{damiano:null}):

\begin{proposition}[Bilinear $L^2$ Strichartz estimate]\label{basl2}  Let $\phi$ be a red wave of frequency $1$, and $\psi$ be a blue wave of frequency $2^k$ for some $k \geq 0$. Then we have
\begin{equation}\label{l2-noloc}
\| \phi \psi \|_{L^2_{t,x}(\R \times \R^n)} \lesssim \M(\phi)^{1/2} \M(\psi)^{1/2}.
\end{equation}
\end{proposition}

This estimate is a model for the bilinear estimates for $X^{s,b}$ spaces, which are of importance in nonlinear wave equations (see e.g. \cite{damiano:null} for a discussion).

In this paper we consider the \emph{inverse problem} for the above estimate: suppose that $\phi, \psi$ are as in Proposition \ref{basl2}, and we have the \emph{lower} bound
\begin{equation}\label{delta}
\| \phi \psi \|_{L^2_{t,x}(\R \times \R^n)} \geq \delta \M(\phi)^{1/2} \M(\psi)^{1/2}
\end{equation}
for some $\delta > 0$.  What can one then conclude about $\phi, \psi$?  Such inverse problems are closely related to the task of obtaining good \emph{profile decompositions} for $\phi, \psi$, which are in turn useful for more refined applications to nonlinear wave equations, see e.g. \cite{kv} for a discussion.  Our primary reason for pursuing this question is that it will have direct application to the global regularity problem for wave maps, and specifically to the large energy perturbation theory of such maps; see \cite{tao:heatwave4}.

\subsection{The equal-frequency case}

To motivate the main results, let us first consider the simpler \emph{equal frequency case} $k=0$.  In this case, we have the following $L^p$ estimate:

\begin{proposition}[Bilinear $L^p$ Strichartz estimate, equal-frequency case]\label{bilp}\cite{wolff:cone}, \cite{tao:cone}  Let $\phi$ be a red wave of frequency $1$, and $\psi$ be a blue wave of frequency $1$.  Then we have
\begin{equation}\label{lp-noloc}
\| \phi \psi \|_{L^p_{t,x}(\R \times \R^n)} \lesssim \M(\phi)^{1/2} \M(\psi)^{1/2}
\end{equation}
for all $p \geq \frac{n+3}{n+1}$.
\end{proposition}

This estimate, first conjectured by Machedon and Klainerman, was established for $p > \frac{n+3}{n+1}$ by Wolff\cite{wolff:cone} (with a constant depending on $p$), with the endpoint $p=\frac{n+3}{n+1}$ being established subsequently by the author in \cite{tao:cone}.  Earlier partial results in this direction (in the important two-dimensional case $n=2$, and with $\frac{n+3}{n+1}$ being replaced by an exponent $2-c$ for some $c>0$) were obtained by Bourgain \cite{borg:cone} and subsequently by Vargas and the author \cite{tv:cone1}.  The exponent $\frac{n+3}{n+1}$ is best possible (see e.g. \cite{tv:cone1} for a counterexample beyond this exponent).  For the purposes of this paper, though, one could replace $\frac{n+3}{n+1}$ by any other exponent strictly less than $2$.

By combining this proposition with H\"older's inequality (interpolating $L^2_{t,x}$ between $L^p_{t,x}$ and $L^\infty_{t,x}$ for some $\frac{n+3}{n+1} \leq p < 2$, and using Bernstein's inequality to bound $\phi,\psi$ in $L^\infty$), we obtain

\begin{corollary}[Concentration at a point, equal frequency case]\label{concpt}  Let $\phi$ be a red wave of frequency $1$, and $\psi$ be a blue wave of frequency $1$, with the normalisation $\M(\phi) = \M(\psi) = 1$.  Suppose that \eqref{delta} holds for some $0 < \delta \lesssim 1$.  Then there exists a point $(t_0,x_0)$ in spacetime such that $|\phi(t_0,x_0)|, |\psi(t_0,x_0)| \gtrsim \delta^{O(1)}$, where we use $O(1)$ to denote a quantity which is $\lesssim 1$.
\end{corollary}

From Bernstein's inequality we also know that $|\nabla_{t,x} \phi|, |\nabla_{t,x} \psi| \lesssim 1$ under the hypotheses in Corollary \ref{concpt}, which implies that the concentration bound $|\phi(t,x)|, |\psi(t,x)| \gtrsim \delta^{O(1)}$ occurs not only at a single point $(t_0,x_0)$, but in fact on a spacetime cube $Q$ of side-length $\gtrsim \delta^{O(1)}$ centred around that point.  On the other hand, Strichartz estimates such as
$$
\| \phi \|_{L^{\frac{2(n+1)}{n-1}}_{t,x}(\R \times \R^n)} \lesssim \M(\phi)$$
(see \cite{strichartz:restrictionquadratic}) tell us that the number of disjoint such cubes is at most $\delta^{-O(1)}$.  Pursuing this idea soon leads to

\begin{corollary}[Profile decomposition, equal frequency case]\label{prof}  Let $\phi$ be a red wave of frequency $1$, and let $0 \leq \delta \lesssim 1$.  Then there exists a family ${\mathcal Q}$ of spacetime cubes $Q$ of size $\gtrsim \delta^{O(1)}$ and cardinality $\lesssim \delta^{-O(1)}$ such that
$$
\| \phi \psi \|_{L^2_{t,x}(\R \times \R^n \backslash \bigcup_{Q \in {\mathcal Q}} Q )} \lesssim \delta \M(\phi)^{1/2} \M(\psi)^{1/2}
$$
for all blue waves $\psi$ of frequency $1$.
\end{corollary}

We leave the proof of this corollary to the reader (and we will prove a more general statement below).

Informally, Corollary \ref{prof} asserts that every red wave $\phi$ has a small exceptional set outside of which one can improve the standard bilinear estimate \eqref{l2-noloc} by any specified parameter $\delta$; this exceptional set should be thought of as the set where $\phi$ is ``large''.  Note also that the exceptional set is \emph{universal} in the sense that it does not depend on the blue wave $\psi$; this universality turns out to be important for our applications to wave maps.  To get some sense of this, let us present a consequence of Corollary \ref{prof}:

\begin{corollary}[Fungibility of bilinear $L^2$ Strichartz, equal-frequency case]\label{fungi}  Let $\phi$ be a red wave of frequency $1$, and let $0 \leq \delta \lesssim 1$.  Then one can decompose $\R$ into $\lesssim \delta^{-O(1)}$ intervals $I$ (including, of course, two unbounded intervals) such that
\begin{equation}\label{fisy}
\| \phi \psi \|_{L^2_{t,x}(I \times \R^n)} \lesssim \delta \M(\phi)^{1/2} \M(\psi)^{1/2}
\end{equation}
for all such $I$ and all blue waves $\psi$ of frequency $1$.
\end{corollary}

\begin{proof}  From Bernstein's inequality and \eqref{massif} we have
$$ \| \phi \|_{L^\infty_t L^2_x(I \times \R^2)} \lesssim \M(\phi)^{1/2}; \quad \| \psi \|_{L^\infty_t L^\infty_x(I \times \R^2)} \lesssim \M(\psi)^{1/2}$$
and thus \eqref{fisy} holds whenever $|I| \leq \delta$.  To conclude the argument, we use Corollary \ref{prof} to obtain the family ${\mathcal Q}$ of spacetime cubes, and use that family to partition $\R$ into $\lesssim \delta^{-O(1)}$ ``unexceptional'' intervals $I$ whose spacetime slabs $I \times \R^n$ do not intersect any cube from ${\mathcal Q}$, plus at most $\lesssim \delta^{-O(1)}$ ``exceptional'' intervals $I$ of length at most $\delta$.
\end{proof}

Very informally, this corollary tells us that a large energy red wave behaves ``as if'' it was small energy once one localises in time, at least for the purposes of equal-frequency bilinear $L^2$ Strichartz estimates.  Furthermore the number of time intervals used in this localisation is controlled by the ratio between what one considers ``large energy'' and what one considers ``small energy''.  Such fungibility results will be useful in extending small energy perturbation theory for nonlinear wave equations to the large energy setting.  We remark in passing that it is not difficult to see that Corollary \ref{fungi} breaks down in the $n=1$ case (in which \eqref{l2-noloc} is basically an identity, and no $L^p$ improvement for $p<2$ is possible).  However, we will restrict attention here to the $n \geq 2$ case (and are in fact primarily interested in the case $n=2$).

\subsection{Main results}

For our intended applications to the global regularity problem for wave maps, it turns out that the equal-frequency inverse theory is insufficient; one must also understand the inverse theory for \eqref{l2-noloc} in the unbalanced frequency case $k \gg 1$.  Here, a partial generalisation of Proposition \ref{bilp} is known:

\begin{proposition}[Bilinear $L^p$ Strichartz estimate]\label{bilp2}\cite{tao:cone}  Let $\phi$ be a red wave of frequency $1$, and $\psi$ be a blue wave of frequency $2^k$ for some $k \geq 0$, and let $\eps > 0$.  Then we have
\begin{equation}\label{lp-noloc2}
\| \phi \psi \|_{L^p_{t,x}(\R \times \R^n)} \lesssim_\eps 2^{k(\frac{1}{p}-\frac{1}{2}+\eps)} \M(\phi)^{1/2} \M(\psi)^{1/2}
\end{equation}
for all $\frac{n+3}{n+1} \leq p \leq 2$, where the subscript in $\lesssim_\eps$ means that the implied constant is allowed to depend on $\eps$.
\end{proposition}

The power of $k$ is sharp except for the $\eps$, as we shall shortly see.  (The $\eps$ loss can probably be removed in the non-endpoint case $p > \frac{n+3}{n+1}$, although we do not pursue this matter here.)

However, it turns out that this estimate (or subsequent variants of this estimate, see e.g. \cite{lv}, \cite{lrv}) is not directly able to establish unbalanced-frequency analogues of the above inverse theory, because the critical cases no longer occur when the waves $\phi, \psi$ concentrate on cubes, but rather when they concentrate along \emph{light rays}.  To explain this phenomenon, let us describe (informally) the key example of unbalanced-frequency red-blue wave interactions. We select a direction $\omega \in S^{n-1}$ with $\angle \omega, e_1 \leq \frac{\pi}{8}$, and introduce a $1 \times 2^k$ spacetime tube
\begin{equation}\label{tube}
T_{t_0,x_0,\omega,k} := \{ (t,x): |t-t_0| \leq 2^k; |x - x_0 - \omega (t-t_0)| \leq 1 \}
\end{equation}
oriented along the null direction $(1,\omega)$.  We will refer to such sets as \emph{$1 \times 2^k$ tubes} for short.  We also consider the \emph{infinite tubes} $T_{t_0,x_0,\omega,\infty}$, defined in the obvious manner.

It is not difficult to create a blue wave $\psi$ of frequency $2^k$ and mass $\M(\psi)=1$ which is concentrated on the tube \eqref{tube}, in the sense that
$$ \int_{|x - x_0-\omega(t-t_0)| \leq 1}  |\psi(t,x)|^2\ dx \gtrsim 1$$
for all $t$ with $|t-t_0| \leq 2^k$.  The basic idea is to select $\psi$ with Fourier support on the sector $\{ (-|\xi|, \xi): 2^k \leq |\xi| \leq 2^{k+1}, \angle(\xi, \omega) \leq 2^{-k} \}$ (one could also take the smaller region $\{ (-|\xi|, \xi): 2^k \leq |\xi| \leq 2^k+1, \angle(\xi, \omega) \leq 2^{-k} \}$ for another example); we omit the details.

The red wave $\phi$ only propagates in directions transverse to $(1,\omega)$ and so cannot send all of its energy along the tube $T_{t_0,x_0,\omega,k}$.  However, it can still have a substantial presence in this tube as follows.  One can cover $T_{t_0,x_0,\omega}$ by about $O(2^k)$ unit cubes $Q$.  For each such cube $Q$, one can find a red wave $\phi_Q$ of energy $\M(\phi_Q)=1$ which concentrates in $Q$ in the sense that $|\phi_Q(t,x)| \gtrsim 1$ for all $(t,x) \in Q$.  If one then sets $\phi := \sum_Q \epsilon_Q c_Q \phi_Q$, where $c_Q$ are coefficients with $\sum_Q |c_Q|^2 = 1$, and $\epsilon_Q = \pm 1$ are iid signs, then an application of Khintchine's inequality\footnote{Actually, there is already enough orthogonality here that the random signs $\epsilon_Q$ are not needed, provided that one supplies a sufficient amount of spacing between the $Q$.} shows that \eqref{l2-noloc} is essentially sharp for this choice of $\phi, \psi$ (and also shows that the exponent of $k$ in \eqref{lp-noloc} cannot be significantly improved).  Observe in this example that the red wave $\phi$ is not concentrated in a single cube $Q$, but can now be dispersed along many cubes intersecting a given tube \eqref{tube}; in particular, some modification to Corollary \ref{concpt} or Corollary \ref{prof} is needed to generalise to the unbalanced-frequency case.

With this motivation, we can now state our main theorem.

\begin{theorem}[Profile decomposition]\label{profile2}  Let $0 < \delta < 1$, let $\phi$ be a red wave of frequency $1$. Then there exists a collection $(T_\beta)_{\beta \in B}$ of infinite tubes \eqref{tube} with cardinality at most $O(\delta^{-O(1)})$ such that
$$ \| \phi \psi \|_{L^2((\R \times \R^n) \backslash \bigcup_{\beta \in B} T_\beta)} \lesssim \delta \M(\phi)^{1/2} \M(\psi)^{1/2}$$
for all $k \geq 0$ and all blue waves $\psi$ of frequency $2^k$.
\end{theorem}

As a corollary, we can generalise Corollary \ref{fungi} to the imbalanced frequency case:

\begin{corollary}[Fungibility of bilinear $L^2$ Strichartz]\label{fungi-2}  Let $\phi$ be a red wave of frequency $1$, and let $0 \leq \delta \lesssim 1$.  Then one can decompose $\R$ into $\lesssim \delta^{-O(1)}$ intervals $I$ such that
\begin{equation}\label{fisy-2}
\| \phi \psi \|_{L^2_{t,x}(I \times \R^n)} \lesssim \delta \M(\phi)^{1/2} \M(\psi)^{1/2}
\end{equation}
for all such $I$ and all blue waves $\psi$ of frequency $2^k$ for some $k \geq 0$.
\end{corollary}

\begin{proof} We may normalise $\M(\phi)=1$.  Let $(T_\beta)_{\beta \in B}$ be as in Theorem \ref{profile2}.  A standard $TT^*$ argument (exploiting the fact that the red wave $\phi$ propagates in directions transverse to the infinite tubes $T_\beta$) shows that
$$ \| \phi \|_{L^2_t L^\infty_x(T_\beta)} \lesssim 1$$
for all $T_\beta$; summing this in $\beta$, we obtain
$$ \int_\R \sum_{\beta \in B} \| \phi 1_{T_\beta}(t) \|_{L^\infty_x(\R^n)}^2\ dt  \lesssim \delta^{-O(1)}.$$
Thus we can partition $\R$ into $\lesssim \delta^{-O(1)}$ intervals such that
$$ \int_I \sum_{\beta \in B} \| \phi 1_{T_\beta}(t) \|_{L^\infty_x(\R^n)}^2\ dt  \lesssim \delta^2$$
or in other words
$$ (\sum_{\beta \in B} \| \phi \|_{L^2_t L^\infty_x(T_\beta \cap (I \times \R^n))}^2)^{1/2} \lesssim \delta.$$
Now let $\psi$ be a blue wave of frequency $2^k$ for some $K \geq 0$.
$$ \| \psi \|_{L^\infty_t L^2_x(I \times \R^n)} = \M(\psi)^{1/2}$$
we see that
$$ (\sum_{\beta \in B} \| \phi \psi \|_{L^2_t L^\infty_x(T_\beta \cap (I \times \R^n))}^2)^{1/2} \lesssim \delta \M(\psi)^{1/2}$$
and hence
$$ \| \phi \psi \|_{L^2_t L^\infty_x(\bigcup_{\beta \in B} T_\beta \cap (I \times \R^n))} \lesssim \delta \M(\psi)^{1/2}.$$
The claim now follows from Theorem \ref{profile2} and the triangle inequality.
\end{proof}

In \cite{tao:heatwave4}, we will use Theorem \ref{profile2} to establish a more general version of Corollary \ref{fungi-2}, in which $\phi, \psi$ solve an inhomogeneous wave equation rather than the free wave equation, and have a more general frequency support.  This will then be used to establish a large energy perturbation theory for wave maps which only involves a bounded number of time intervals, which will be needed in order to create minimal-energy blowup solutions.

\subsection{Acknowledgements}

The author thanks Jacob Sterbenz for useful discussions.  The author is supported by NSF grant DMS-0649473, a grant from the Macarthur Foundation, and the NSF Waterman award.

\section{A key decomposition and its consequences}

Throughout this paper, a \emph{cube} denotes a cube in spacetime $\R \times \R^n$ with sides parallel to the axes.  If $Q$ is a cube and $0 < c < 1$, define $(1-c)Q$ to be the cube with the same centre as $Q$ but with $1-c$ of the sidelength.

We need the useful technical notion of the \emph{margin} of a red wave.

\begin{definition}[Margin]\cite{tao:cone}  If $\phi$ is a red wave of frequency $2^k$, we define the \emph{margin} $\margin(\phi)$ of $\phi$ to be the quantity
$$ \margin(\phi) := \dist( 2^{-k} \supp(\tilde \phi), \partial (\Sigma_\red \cap \Sigma_0) )$$
where $\partial (\Sigma_\red \cap \Sigma_0)$ is the topological boundary of $\Sigma_\red \cap \Sigma_0$ in the cone $\Sigma$.  
\end{definition}

We now recall a key decomposition from \cite{tao:cone} which will underlie the results here.

\begin{proposition}[Decomposition]\label{decomp}  Let $Q$ be a cube of sidelength $R \gg 1$ (i.e. $R \geq C$ for some sufficiently large absolute constant $C$).  Let $\phi$ be a red wave of frequency $1$ and margin at least $R^{-1/2}$, and let $\psi$ be a blue wave of frequency $2^k$ for some $k \geq 0$.

Decompose $Q$ into $2^{n+1}$ subcubes $Q_1,\ldots,Q_{2^{n+1}}$ of sidelength $R/2$, and let $0 < c < 1/2$.  Then there exists red waves $\phi_i$ of frequency $1$ for $i=1,\ldots,2^{n+1}$ with the following properties:
\begin{itemize}
\item[(i)] (Bessel inequality) We have
\begin{equation}\label{bessel}
\sum_{i=1}^{2^{n+1}} \M(\phi_i) \leq (1 + O(c)) \M(\phi).
\end{equation}
\item[(ii)] (Margin bound)  For every $1 \leq j \leq 2^{n+1}$, we have
\begin{equation}\label{margin}
\margin(\phi_i) \geq \margin(\phi) - O( R^{-1/2} ).
\end{equation}
\item[(iii)] (Approximation)  For every \emph{distinct} $1 \leq i, i' \leq 2^{n+1}$, we have
\begin{equation}\label{approx}
 \| |\phi - \phi_i| |\psi| \|_{L^2_{t,x}( (1-c) Q_{i'} )} \lesssim c^{-O(1)} R^{-(n-1)/4} \M(\phi)^{1/2} \M(\psi)^{1/2}.
\end{equation}
\end{itemize}
\end{proposition}

\begin{proof} See \cite[Proposition 15.1]{tao:cone} (specialising $C_0=1$, $j=0$, $j'=k$).  Note that the mass $\M(\phi)$ of a wave is denoted $\E(\phi)$ in \cite{tao:cone}.
\end{proof}

This gives us a preliminary localisation result.

\begin{proposition}[Localisation]\label{local}  Let $Q$ be a cube of sidelength $R > 0$.  Let  $\phi$ be a non-trivial red wave of frequency $1$, and let $\psi$ be a non-trivial blue wave of frequency $2^k$ for some $k \geq 0$.  Let $\Omega \subset \R \times \R^n$, and suppose that
$$ \| \phi \psi \|_{L^2(Q \cap \Omega)} \gtrsim \delta \M(\phi)^{1/2} \M(\psi)^{1/2}$$
for some $0 < \delta < 1$. Then there exists a cube $q$ of sidelength $1$ such that
\begin{equation}\label{psiq}
 \| \psi \|_{L^2(q \cap \Omega)} \gtrsim \delta^{O(1)} \M(\psi)^{1/2}.
\end{equation}
\end{proposition}

\begin{proof}  By decomposing $\phi$ into boundedly many pieces and applying some Lorentz transformations, we may assume that $\phi$ has margin at least $1/10$ (say).
We can then normalise $\M(\phi)=\M(\psi)=1$.  
Let $J$ be the largest integer such that $2^{-J} R \geq C \delta^{-C}$, where $C$ is a large absolute constant.  We may assume that $J \geq 0$, since the claim is trivial otherwise (just partition $Q$ into unit cubes, use the pigeonhole principle, and use Bernstein's inequality and \eqref{massif} to bound $\phi$ in $L^\infty$).

It is convenient to replace $Q$ by a slightly smaller set.
For any $0 \leq j \leq J$, let ${\mathcal Q}_{j}$ be the partition of $Q$ into $2^{(n+1)j}$ cubes of sidelength $R/2^j$.  Let $X(Q) \subset Q$ be the set
$$ X(Q) := \bigcap_{j=0}^J \bigcup_{q \in {\mathcal Q}_{j}} (1 - c_{j}) q$$
where $c_{j} := \eps 2^{-\eps(J-j)}$, and $\eps > 0$ is a small constant to be chosen later.  A short calculation shows that $|X(Q)| \sim |Q|$ if $\eps$ is small enough, so by averaging over translations of $X(Q)$ (and replacing $Q$ by a translate if necessary) we may assume that
$$ \| \phi \psi \|_{L^2(X(Q) \cap \Omega)} \gtrsim \delta.$$
Now we apply Proposition \ref{decomp} repeatedly (using the parameters $c_{j}$ at scale $R/2^j$).  This gives us a collection of red waves $\phi_q$ for each $q \in {\mathcal Q}_{j}$ and $0 \leq j \leq J$, with the following properties:
\begin{itemize}
\item[(i)] (Initial condition) $\phi_Q = \phi$.
\item[(ii)] (Bessel inequality) For each $0 \leq j \leq J-1$ and $q \in {\mathcal Q}_{j}$ we have
\begin{equation}\label{bessel-2}
\sum_{q' \in {\mathcal Q}_{j+1}: q' \subset q} \M( \phi_{q'} ) \leq (1 + O_\eps( 2^{-\eps(J-j)} )) \M(\phi_q).
\end{equation}
\item[(iii)] (Margin bound)  For each $0 \leq j \leq J-1$, $q \in {\mathcal Q}_{j}$, $q' \in {\mathcal Q}_{j+1}$ with $q' \subset q$ we have
\begin{equation}\label{margin-2}
\margin(\phi_{q'}) \geq \margin(\phi_q) - O( (R/2^j)^{-1/2} )
\end{equation}
(and so by induction, $\margin(\phi_q) \geq 1/20$).
\item[(iv)] (Approximation)  For each $0 \leq j \leq J-1$, $q \in {\mathcal Q}_{j}$, and distinct $q', q'' \in {\mathcal Q}_{j+1}$ with $q', q'' \subset q$ we have
\begin{equation}\label{approx-2}
 \| |\phi_q - \phi_{q'}| |\psi| \|_{L^2_{t,x}( (1-c_{j+1}) q'' )} \lesssim_\eps 2^{O(\eps(J-j))} (R/2^j)^{-(n-1)/4} \M(\phi_q)^{1/2}.
\end{equation}
\end{itemize}

Iterating \eqref{bessel-2} we see that
\begin{equation}\label{jump}
 \sum_{q \in {\mathcal Q}_j} \M( \phi_q ) \lesssim 1
\end{equation}
for all $0 \leq j \leq J$.  Also, if we let
$$ F_j := (\sum_{q \in {\mathcal Q}_j} \| \phi_q \psi \|_{L^2_{t,x}( q \cap X(Q) \cap \Omega )}^2)^{1/2}$$
then we see from square-summing \eqref{approx-2}, \eqref{jump}, and the triangle inequality that
$$ F_{j+1} = F_j + O_\eps( 2^{O(\eps(J-j))} (R/2^j)^{-(n-1)/4} )$$
for all $0 \leq j \leq J$.  Taking $\eps$ small enough, we can sum this in $j$ and conclude that
$$ F_J = F_0 + O( (R/2^J)^{-(n-1)/4} ).$$
On the other hand, from hypothesis we have $F_0 \gtrsim \delta$.  From definition of $J$, we conclude that
$$ F_J \gtrsim \delta,$$
and thus from \eqref{jump} and the pigeonhole principle, there exists $q \in {\mathcal Q}_J$ such that $\phi_q$ is non-zero and
$$ \| \phi_q \psi \|_{L^2_{t,x}( q \cap X(Q) \cap \Omega )} \gtrsim \delta \M( \phi_q )^{1/2}.$$
But from \eqref{massif} and Bernstein's inequality we have $\| \phi_q \|_{L^\infty_{t,x}} \lesssim \M( \phi_q )^{1/2}$, and thus
$$ \| \psi \|_{L^2_{t,x}( q )} \gtrsim \delta.$$
But $q$ has sidelength $O(\delta^{-O(1)})$, so the claim follows by covering $q$ by unit cubes.
\end{proof}

Now we analyse the cubes $q$ that obey the property \eqref{psiq}.  We begin with a covering lemma.

\begin{lemma}[Covering lemma]\label{cover}  Let $(T_\beta)_{\beta \in B}$ be a (possibly infinite) family of $1 \times 2^k$ tubes $T_\beta = T_{t_\beta,x_\beta,\omega_\beta,k}$ of the form \eqref{tube} with $t_\beta=0$, which are \emph{separated} in the sense that $|x_\beta-x_{\beta'}| + 2^k |\omega_\beta - \omega_{\beta'}| \gtrsim 1$ for any two distinct tubes $T_\beta, T_{\beta'}$.  For each $\beta \in B$ let $c_\beta \geq 0$ be a number such that
$$ \sum_{\beta \in B} c_\beta \leq 1.$$
Let $0 < \delta \leq 1$.  Then one has
$$\sum_{\beta \in B} c_\beta 1_{T_\beta}(t,x) \leq \delta$$
for all $(t,x)$ outside of a union of at most $O(\delta^{-3})$ $1 \times 2^k$ tubes.
\end{lemma}

\begin{proof}  We perform the following greedy algorithm to build some points $(t_1,x_1), (t_2,x_2), \ldots, (t_k,x_k)$ and disjoint sets $B_1, \ldots, B_k \subset B$.
\begin{itemize}
\item (Step 0) Initialise $k=0$.
\item (Step 1) If one has
\begin{equation}\label{sunbeam}
\sum_{\beta \in B \backslash (B_1 \cup\ldots \cup B_{k})} c_\beta 1_{T_\beta}(t,x) \leq \delta/2
\end{equation}
for all $(t,x)$, then {\tt STOP}. 
\item (Step 2) Otherwise, we can find a $(t_{k+1},x_{k+1})$ such that \eqref{sunbeam} fails.  Set
$$B_{k+1} := \{ \beta \in B \backslash (B_1 \cup\ldots \cup B_{k}): (t_{k+1},x_{k+1}) \in T_\beta \},$$
increment $k$ to $k+1$, and return to Step 1.
\end{itemize}

Observe that each time Step 2 is invoked, the quantity $\sum_{\beta \in B_1 \cup \ldots \cup B_k} c_\beta$ increases by at least $\delta/2$.  Thus the algorithm must terminate in at most $2/\delta$ steps.  By construction, \eqref{sunbeam} holds for all $(t,x)$, so it suffices to show that
$$\sum_{\beta \in B_1 \cup \ldots \cup B_k} c_\beta 1_{T_\beta}(t,x) \leq \delta/2$$
for all $(t,x)$ outside  of a union of at most $O(\delta^{-3})$ $1 \times 2^k$ tubes.  By the triangle inequality, it suffices to show that for each $1 \leq j \leq k$, one has
\begin{equation}\label{bjj}
\sum_{\beta \in B_j} c_\beta 1_{T_\beta}(t,x) \leq \delta^2/4
\end{equation}
for all $(t,x)$ outside of a union of at most $O(\delta^{-2})$ $1 \times 2^k$ tubes.

Fix $j$.  The tubes $T_\beta$ for $\beta$ in $B_j$ all pass through a single point $(t_j,x_j)$, and so are almost entirely determined by the direction $\omega_\beta \in S^{n-1}$ of those tubes.

Partition $S^{n-1}$ into a dyadic grid (e.g. by first breaking up $S^{n-1}$ into two hemispheres, and identifying those hemispheres with a dyadic cube by some bilipschitz map).  Call a square $q$ in that grid \emph{large} if $\sum_{\beta \in B_j: \omega_\beta \in q} c_\beta$ exceeds $\delta'$, where $\delta' > 0$ is a small parameter to be chosen later.  Consider the large squares which are minimal with respect to set inclusion.  Then (by the dyadic nature of the grid) these squares are disjoint; since $\sum_{\beta \in B} c_\beta \leq 1$, there are at most $1/\delta'$ such minimal large squares.

For each minimal large square $q$, let $T_q$ be the $C \times C 2^k$-tube centred at $(t_j,x_j)$ and oriented in the direction $(1,\omega_q)$, where $\omega_q$ is the centre of $q$ and $C$ is a sufficiently large constant.  We claim that (with $\delta$ equal to a sufficiently small multiple of $\delta^2$), \eqref{bjj} holds for all $(t,x)$ outside of all the $T_q$, and also obeying $|t-t_j|+|x-x_j| \geq C$; this will establish the lemma, since we can cover each $T_q$ by $O(1)$ $1 \times 2^k$ tubes.

To verify the claim, suppose that $(t,x)$ is such that $|t-t_j|+|x-x_j| \geq C$ for some large $C$ but obeys
$$ \sum_{\beta \in B_j} c_\beta 1_{T_\beta}(t,x) > \delta^2/4.$$
We may assume that $|t-t_j| \geq 10$, as the sum on the left must be empty otherwise.
Write $\omega := (x-x_j)/(t-t_j)$.  Observe from elementary geometry that in order for $T_\beta$ to contain $(t,x)$, $\omega_\beta$ must lie within $O( 1 / |t-t_j| )$ of $\omega$, thus
$$ \sum_{\beta \in B_j: |\omega_\beta - \omega| \lesssim 1/|t-t_j|} c_\beta \gtrsim \delta^2.$$
By the pigeonhole principle (and taking $\delta'$ equal to a small multiple of $\delta^2$), this implies the existence of a large square $q$ of sidelength $\sim 1/|t-t_j|$ and within a distance $O(1/|t-t_j|)$ of $\omega$.  This large square must contain a minimal large square $q'$, and elementary geometry then shows that $(t,x)$ lies in $T_q$ if $C$ is large enough, and the claim follows.
\end{proof}

As a consequence, we have

\begin{proposition}[Exceptional tubes]\label{except}  Let $\psi$ be a non-trivial blue wave of frequency $2^k$ for some $k \geq 0$, and let $\delta > 0$.  Then there exists $O(\delta^{-O(1)})$ $1 \times 2^k$ tubes, such that
$$ \| \psi \|_{L^2(q)} \leq \delta \M(\psi)^{1/2}$$
for any cube $q$ of unit sidelength not touching one of these tubes.
\end{proposition}

\begin{proof}  By partitioning $\psi$ into a bounded number of pieces and applying some slight Lorentz transforms, we may assume that $\psi$ has margin at least $1/10$ (say).  We may normalise $\M(\psi)=1$.  Call a cube $q$ of unit sidelength \emph{bad} if
$$ \| \psi \|_{L^2(q)} > \delta.$$

We first perform a $TT^*$ analysis, analogous to that used to prove Strichartz estimates.
Suppose we can find $N$ bad cubes $q_1,\ldots,q_N$, whose centres $(t_1,x_1),\ldots,(t_N,x_N)$ have the separation property $|t_i-t_j| \geq T$ for all distinct $1 \leq i,j \leq N$ and some $T \gg 2^k$.  By the pigeonhole principle and duality, for each $1 \leq j \leq N$ we can find a time $t'_j = t_j + O(1)$ and a function $f_j$ of $L^2_x$ norm $1$ supported in the region $\{ x: |x-x_j| \lesssim 1 \}$ such that
$$ \hbox{Re} \langle \psi(t'_j), f_j \rangle_{L^2_x} \gtrsim \delta.$$
By Fourier analysis, one can write the left-hand side as
$$ \hbox{Re} \langle \psi(0), U(-t'_j) f_j \rangle_{L^2_x} \gtrsim \delta.$$
where  $U(-t)$ is a Fourier multiplier with symbol $\eta(\xi/2^k) e^{2\pi it|\xi|}$, where $\eta$ is a smooth Littlewood-Paley-type cutoff to the annulus $\{ \xi \in \R^n: |\xi| \sim 1 \}$.  Summing in $j$ and then using Cauchy-Schwarz, we conclude that
$$ \| \sum_{j=1}^N U(-t'_j) f_j \|_{L^2_x(\R^n)} \gtrsim N \delta.$$
Squaring this, we see that
\begin{equation}\label{nij}
 \sum_{1 \leq i,j \leq N} |\langle U(-t'_i) f_i, U(-t'_j) f_j \rangle_{L^2_x}| \gtrsim N^2 \delta^2.
\end{equation}
On the other hand, from Plancherel's theorem we have
$$ |\langle U(-t'_j) f_j, U(-t'_j) f_j \rangle_{L^2_x}| \lesssim 1$$
and when $i \neq j$, the decay of the convolution kernel of $U(-t'_i)^* U(-t'_j)$ (which can be easily computed using stationary phase) and the localisation of the $f_i, f_j$ give the dispersive bound
$$ |\langle U(-t'_i) f_i, U(-t'_j) f_j \rangle_{L^2_x}| \lesssim (T/2^k)^{-(n-1)/2}.$$
We thus conclude that
$$ \sum_{1 \leq i,j \leq N} |\langle U(-t'_i) f_i, U(-t'_j) f_j \rangle_{L^2_x}| \lesssim N + (T/2^k)^{-(n-1)/2} N^2.$$
This will contradict \eqref{nij} if $N > C\delta^{-C}$ and $T > C \delta^{-C} 2^k$ for some sufficiently large $C$.  Thus we see that we cannot find more than $O( \delta^{-O(1)})$ bad cubes whose centres have time coordinates separated by more than $C \delta^{-C} 2^k$.  Applying a greedy algorithm, this implies that we can cover the union of all the bad cubes by at most $O(\delta^{-O(1)})$ slabs $I \times \R^n$ of duration $O( \delta^{-O(1)} 2^k )$.  By subdividing these slabs further we may assume that each time interval $I$ has length at most $2^k$.  It thus suffices to show that the bad cubes in any time interval $I$ of length $2^k$ can be covered by $O( \delta^{-O(1)} )$ $1 \times 2^k$ tubes.  By time translation we may take $I = [0,2^k]$.

The Fourier transform of $\psi$ is supported on $\Sigma_{\blue} \cap \Sigma_k$.  We cover this set by about $O(2^{nk})$ subsets $\alpha$ of the form $\alpha = \{ (-|\xi|, \xi): |\xi - \xi_\alpha| \lesssim 1 \}$, where $|\xi_\alpha| \sim 2^k$ and $\angle \xi_0, e_1 \leq \pi/8$.  Using a partition of unity, we can then write $\psi = \sum_\alpha \psi_\alpha$, where each $\psi_\alpha$ is a blue wave with Fourier support on $\alpha$, and 
$$ \sum_\alpha \M(\psi_\alpha) \lesssim 1.$$

Now let $\eta$ be a Schwartz function, positive on the ball $\{ x: |x| \lesssim 1 \}$, whose Fourier transform is supported in a ball $\{ \xi: |\xi| \lesssim 1 \}$.  Then for any unit cube $q$ with centre $(t_0,x_0)$, we can estimate
$$ \| \psi \|_{L^2_{t,x}(q)}^2 \lesssim \int_{t = t_0+O(1)} \int_{\R^d} |\psi(t,x)|^2 \eta(x-x_0)^2\ dx dt.$$
If we expand $\psi = \sum_\alpha \psi_\alpha$, we can write this as
$$ \| \psi \|_{L^2_{t,x}(q)}^2 \lesssim \sum_{\alpha,\beta} \int_{t = t_0+O(1)} \langle \eta(\cdot-x_0) \psi_\alpha(t), \eta(\cdot-x_0) \psi_\beta(t) \rangle_{L^2_x}\ dt.$$
Observe from Fourier analysis that the inner product vanishes unless $\alpha, \beta$ are within $O(1)$ of each other.  Discarding all the vanishing terms, and applying Cauchy-Schwarz followed by Schur's test, we conclude that
$$ \| \psi \|_{L^2_{t,x}(q)}^2 \lesssim \sum_{\alpha} \int_{t = t_0+O(1)} \| \eta(\cdot-x_0) \psi_\alpha(t)\|_{L^2_x(\R^n)}^2.$$
Let $F_\alpha(t_0,x_0)$ denote the quantity
$$ F_\alpha(t_0,x_0) := \| \eta(\cdot-x_0) \psi_\alpha(t_0)\|_{L^2_x(\R^n)}^2,$$
thus it suffices to establish the pointwise estimate
$$ \sum_{\alpha} F_\alpha(t_0,x_0) \leq c \delta^2$$
for all $(t_0,x_0) \in [0,2^k] \times \R^n$ outside of the union of at most $O(\delta^{-O(1)})$ $1 \times 2^k$ tubes, where $c > 0$ is a sufficiently small constant.

Let us first understand the situation for a fixed $\alpha$.  From the Fourier support of $\psi_\alpha(t_0)$, and the fact that $t_0 = O(2^k)$, we can express $\psi_\alpha(t_0)$ as the convolution of $\psi_\alpha(0)$ with a kernel $K_{\alpha,t_0}(x)$ which is bounded pointwise by 
$$ |K_{\alpha,t_0}(x)| \lesssim (1 + |x - \omega_\alpha t_0|)^{-100n},$$
where $\omega_\alpha := \xi_\alpha/|\xi_\alpha|$ is the direction of the centre of $\alpha$.  From this we see that
$$ F_\alpha(t_0,x_0) \lesssim \int_{\R^n} |\psi_\alpha(t_0,x)|^2 (1 + |x-x_0-\omega_\alpha t_0|)^{-50n}\ dx.$$
If we then cover $[0,2^k] \times \R^n$ by boundedly overlapping $1 \times 2^k$ tubes $(T_\beta)_{\beta \in B_\alpha}$ with $t_0=0$ and oriented in the direction $(1,\omega_\alpha)$, and separated in the sense of Lemma \ref{cover}, we see that we have the pointwise bound
$$ F_\alpha \leq \sum_{\beta \in B_\alpha} c_\beta 1_{T_\beta}$$
for some constants $c_\beta \geq 0$ obeying the bound
$$ \sum_{\beta \in B_\alpha} c_\beta  \lesssim \M(\psi_\alpha).$$
Summing in $\alpha$, we obtain 
$$ \sum_{\alpha \in F_\alpha} \leq \sum_{\beta \in B} c_\beta 1_{T_\beta}$$
for some collection of $1 \times 2^k$ tubes $(T_\beta)_{\beta \in B}$ with $t_0=0$ and separated in the sense of Lemma \ref{cover}, and with
$$ \sum_{\beta \in B} c_\beta  \lesssim 1.$$
The claim now follows from Lemma \ref{cover}.
\end{proof}

Combining Proposition \ref{local} with Proposition \ref{except} we obtain

\begin{corollary}[Exceptional tubes, II]\label{local2}  Let $0 < \delta < 1$, let $\phi$ be a red wave of frequency $1$, and let $\psi$ be a nonblue wave of frequency $2^k$ for some $k \geq 0$.  Then there exists a collection $(T_\beta)_{\beta \in B}$ of $1 \times 2^k$ tubes with cardinality at most $O(\delta^{-O(1)})$ such that
$$ \| \phi \psi \|_{L^2((\R \times \R^d) \backslash \bigcup_{\beta \in B} T_\beta)} \lesssim \delta \M(\phi)^{1/2} \M(\psi)^{1/2}.$$
\end{corollary}

\begin{proof} By an approximation argument we may assume that $\phi, \psi$ are non-trivial, smooth and rapidly decreasing in space; the frequency separation of $\phi,\psi$ then ensures that $\phi\psi$ is also rapidly decreasing in time.  This lets us replace spacetime $\R \times \R^d$ by a sufficiently large cube $Q$ (we have no upper bound on the size $R$ of this cube, but our bounds will not involve $R$).

Let $\delta' > 0$ be chosen later.  By Proposition \ref{except} we have
$$ \| \psi \|_{L^2(q)} \leq \delta' \M(\psi)^{1/2}$$
for all cubes $q$ of unit length not touching one of $O((\delta')^{-O(1)})$ $1 \times 2^k$ tubes.  By dilating each of these tubes by a bounded amount, one can thus find a set $\Omega$ which is the complement of the union of $O((\delta')^{-O(1)})$ $1 \times 2^k$ tubes such that
$$ \| \psi \|_{L^2(q \cap \Omega)} \leq \delta' \M(\psi)^{1/2}$$
for all cubes $q$ of unit length.  The claim now follows from Proposition \ref{local} if we set $\delta' = C^{-1} \delta^C$ for a sufficiently large $C$.
\end{proof}

\section{Removing the exceptional rays}

Corollary \ref{local2} is close to Theorem \ref{profile2}, but differs from it in that the tubes $T_\beta$ depend on $\psi$.  We now work to make these tubes independent of $\psi$.  It will suffice to prove the following statement:

\begin{proposition}[Exceptional rays]\label{tinny}  Let $0 < \delta < 1$, and let $\phi$ be a red wave of frequency $1$.  Then there exists a collection $(T_\beta)_{\beta \in B}$ of infinite tubes with cardinality at most $O(\delta^{-O(1)})$, such that if one lets $\Omega := (\R \times \R^n) \backslash \bigcup_{\beta \in B} T_\beta$, then
$$ \| \phi \|_{L^2_t L^\infty_x(\Omega \cap T)} \leq \delta \M(\phi)^{1/2}$$
for all infinite tubes $T$.
\end{proposition}

Indeed, suppose we had Proposition \ref{tinny}.  Let $\delta, \phi$ be as in Theorem \ref{profile2}.  We apply Proposition \ref{tinny} with $\delta$ replaced by $\delta' := C^{-1} \delta^C$ for a sufficiently large $C$, and obtain $(T_\beta)_{\beta \in B}$ and $\Omega$ obeying the conclusions of that proposition.  Next, for any $k \geq 0$ and blue wave $\psi$ of frequency $2^k$, we apply Corollary \ref{local2} to obtain another collection $(T'_{\beta'})_{\beta' \in B'}$ of $1 \times 2^k$ tubes of cardinality at most $O(\delta^{-O(1)})$ such that
$$ \| \phi \psi \|_{L^2_{t,x}((\R \times \R^n) \backslash \bigcup_{\beta' \in B'} T'_{\beta'})} \lesssim \delta \M(\phi)^{1/2} \M(\psi)^{1/2}$$
and in particular
$$ \| \phi \psi \|_{L^2_{t,x}(\Omega \backslash \bigcup_{\beta' \in B'} T'_{\beta'})} \lesssim \delta \M(\phi)^{1/2} \M(\psi)^{1/2}.$$
On the other hand, for each $T'_{\beta'}$, we see from Proposition \ref{tinny} that
$$ \| \phi \|_{L^2_t L^\infty_x(\Omega \cap T'_{\beta'})} \lesssim \delta' \M(\phi)^{1/2}$$
and thus by \eqref{massif}
$$ \| \phi\psi \|_{L^2_t L^2_x(\Omega \cap T'_{\beta'})} \lesssim \delta' \M(\phi)^{1/2} \M(\psi)^{1/2}.$$
Summing in $\beta'$ (and choosing $\delta'$ appropriately) we obtain
$$ \| \phi \psi \|_{L^2_{t,x}(\Omega \cup \bigcup_{\beta' \in B'} T'_{\beta'})} \lesssim \delta \M(\phi)^{1/2} \M(\psi)^{1/2},$$
and so by the triangle inequality
$$ \| \phi \psi \|_{L^2_{t,x}(\Omega)} \lesssim \delta \M(\phi)^{1/2} \M(\psi)^{1/2},$$
and Theorem \ref{profile2} follows.

Thus, the only remaining task is to establish Proposition \ref{tinny}.  By the usual decomposition and Lorentz transform trick, we may assume that $\phi$ has margin at least $1/10$.  We can also assume that $\delta > 0$ is small.  The key proposition is

\begin{proposition}[Concentration implies mass removal]\label{concrem}  Let $0 < \delta \ll 1$ be small.  Suppose that $\phi$ is a red wave of frequency $1$ and mass $\M(\phi) \leq 1$ and margin $\margin(\phi)$ at least $1/20$, and $T$ is an infinite tube such that
$$ \| \phi \|_{L^2_t L^\infty(T)} \geq \delta.$$
Then one can write $\phi = F + (\phi-F)$, where $F$ is a red wave of frequency $1$ and margin $\margin(F) \geq \margin(\phi) - \delta^{10}$ such that
$$ \| F \|_{L^2_t L^\infty(T' \backslash \delta^{-C} T)} \leq \delta^{5}$$
for all infinite tubes $T'$ and some absolute constant $C$, where $\delta^{-C} T$ is the dilation of $T$ by $\delta^{-C}$.  Furthermore, we have the mass decrement property
$$ \M(\phi - F) \leq \M(\phi) - c \delta^3$$
for some absolute constant $c>0$.
\end{proposition}

Indeed, if this proposition held, then by iterating, we could decompose $\phi$ as the sum of $O(\delta^{-3})$ functions $F_1,\ldots,F_m$, each associated to an infinite tube $T_1,\ldots,T_m$ such that
$$ \| F_i \|_{L^2_t L^\infty(T' \backslash \delta^{-C} T_i)} \leq \delta^{5}$$
for all infinite tubes $T'$, plus a remainder $\phi' := \phi - F_1 - \ldots - F_M$ with the property that
$$ \| \phi' \|_{L^2_t L^\infty(T)} \leq \delta.$$
for all infinite tubes $T$.  Proposition \ref{tinny} then follows by covering each of the $\delta^{-C} T_i$ by $O(\delta^{-O(1)})$ infinite tubes, and using the triangle inequality away from these tubes.

It remains to prove Proposition \ref{concrem}.  We may as well take $T = T_{0,0,\omega,\infty}$, thus
$$ (\int_\R \sup_{|x - t \omega_0| \leq 1} |\phi(t,x)|^2\ dt)^{1/2} \geq \delta.$$
By duality, we can find a function $f \in L^2_t(\R)$ with $\|f\|_{L^2_t(\R)} \leq 1$, and a measurable function $x: \R \to \R^n$ with $|x(t)-t \omega_0| \leq 1$ for all $t$, such that
$$ \int_\R \phi(t,x(t)) \overline{f(t)}\ dt \gtrsim \delta.$$
By a limiting argument we may assume that $f$ is smooth and compactly supported, and that $x$ is smooth as well, so that there is no difficulty justifying the manipulations below.  By Fourier analysis, we can rewrite the left-hand side as
$$ \langle \phi(0), F(0) \rangle_{L^2_x}$$
where
\begin{equation}\label{fdef}
 F(t,x) = \int_\R f(t') \int_{\R^n} e^{2\pi i (t-t') |\xi|} e^{2\pi i (x-x(t')) \cdot \xi} \eta(\xi)\ d\xi dt'
\end{equation}
where $\eta = \eta_\delta$ is a bump function supported in the region of frequency space corresponding to red waves of margin $\margin(\phi)-\delta^{10}$ which equals $1$ on the region of frequency space corresponding to red waves of margin $\margin(\phi)$.  Thus $F$ is itself a red wave of margin at least $\margin(\phi)-\delta^{10}$.

Let us compute the mass $\M(F)$.  By Plancherel's theorem, this is equal to
$$ \M(F) = \int_{\R^n} \int_\R |f(t') e^{-2\pi i t' |\xi|} e^{-2\pi i x(t') \cdot \xi}\ dt'|^2 \eta(\xi)^2\ d\xi$$
which can be expanded out as
$$ \int_\R \int_\R f(t) \overline{f(t')} (\int_{\R^n} e^{2\pi i (t-t') |\xi|} e^{-2\pi i (x(t)-x(t')) \cdot \xi} \eta(\xi)^2\ d\xi)\ dt dt'.$$
Now we compute the inner integral.  When $|t-t'|=O(1)$ we can bound this integral crudely by $O(1)$, so suppose that $|t-t'|$ is large.  From the support of $\eta(\xi)$ and the fact that $x(t)$ moves (up to errors of $O(1)$) in a direction $\omega_0$ transverse to red waves, we see that the gradient of the phase has magnitude comparable to $|t-t'|$ on the support of $\eta(\xi)$.  If we integrate by parts once, we obtain a bound of $O(1/|t-t'|)$; if we integrate by parts twice, we obtain $O( 1/(\delta^{10} |t-t'|^2) )$.  Thus we obtain a net bound of
$$ O( \min( 1, 1/(|t-t'|), 1/(\delta^{10} |t-t'|^2) ) )$$
for this integrand, and so from Young's inequality (or Schur's test) and the $L^2$ normalisation of $f$, we conclude that
$$ \M(F) \lesssim \log \frac{1}{\delta}.$$
On the other hand, we have
$$ \langle \phi(0), F(0) \rangle_{L^2_x} \gtrsim \delta$$
and $\M(\phi) \leq 1$, so from the cosine rule we can find a constant $0 < \mu \leq 1$ such that
$$ \M( \phi - \mu F ) \leq \M(\phi) - c \delta^2 / \log \frac{1}{\delta}$$
for some absolute constant $c>0$.  Thus, to finish the proof of Proposition \ref{concrem} (with $F$ replaced byby $\mu F$), it will suffice to show that
$$ \| F \|_{L^2_t L^\infty(T' \backslash \delta^{-C} T)} \leq \delta^{5}$$
for all infinite tubes $T'$ and some absolute constant $C$.

We inspect the kernel
$$ K(t,x,t') := \int_{\R^n} e^{2\pi i (t-t') |\xi|} e^{2\pi i (x-x(t')) \cdot \xi} \eta(\xi)\ d\xi$$
appearing in \eqref{fdef}.  Standard stationary phase estimates show that  this kernel can be bounded in magnitude by
$$ |K(t,x,t')| \lesssim \delta^{-O(1)} (1+|t-t'|)^{-(n-1)/2} (1 + \dist( ((t-t'),(x-x(t'))), S ) )^{-100n}$$
where $S$ is the double cone
$$ S := \{ (-|x|,x): \angle x, e_1 \leq \frac{\pi}{8} \} \cup \{ (|x|,-x): \angle x, e_1 \leq \frac{\pi}{8} \}.$$
Since $x(t') = t \omega_0 + O(1)$, we can rewrite this bound as
\begin{equation}\label{kdef}
|K(t,x,t')| \lesssim \delta^{-O(1)} (1+|t-t'|)^{-(n-1)/2} (1 + \dist( (t,x) - t' (1,\omega_0), S ) )^{-100n}.
\end{equation}
Now let $T'$ be another infinite tube, which we may write as
$$ T' = \{ (t,x): |x - x_1 - \omega_1 t| \leq 1 \}.$$
By \eqref{fdef}, \eqref{kdef}, we can then bound
\begin{align*}
 \sup_{x: (t,x) \in T'} |F(t,x)| &\lesssim \delta^{-O(1)} \int_\R  (1+|t-t'|)^{-(n-1)/2} \\
&\quad (1 + \dist( (t,x_1 + \omega_1 t) - t' (1,\omega_0), S ) )^{-100n} |f(t')|\ dt'.
\end{align*}
The exclusion of $\delta^{-C} T$ forces $|x_1 + \omega_1 t - \omega_0 t| \geq \delta^{-C}/2$, which implies from elementary geometry that either $|t-t'| \geq \delta^{-C}/4$ or 
$\dist( (t,x_1 + \omega_1 t) - t' (1,\omega_0), S ) \gtrsim \delta^{-C}$.  Thus we see that
$$
 \sup_{x: (t,x) \in T'} |F(t,x)|
 \lesssim \delta^{C/2-O(1)}
\int_\R  (1 + \dist( (t,x_1 + \omega_1 t) - t' (1,\omega_0), S ) )^{-50n} |f(t')|\ dt'$$
assuming this restriction.  Thus, to conclude the proof of Proposition \ref{concrem}, it suffices by Schur's test to show that
$$ \int_\R (1 + \dist( (t,x_1 + \omega_1 t) - t' (1,\omega_0), S ) )^{-50n} \ dt' \lesssim 1$$
for all $t$, and dually that
$$ \int_\R (1 + \dist( (t,x_1 + \omega_1 t) - t' (1,\omega_0), S ) )^{-50n} \ dt \lesssim 1$$
for all $t'$.  But these estimates easily follow from the transversality of $S$ to $(1,\omega_0)$, $(1,\omega_1)$.  The proof of Proposition \ref{concrem}, and hence Theorem \ref{profile2}, follows.

\end{document}